\documentclass[graphicx]{article}
\usepackage{amsthm,amsfonts, amsbsy, amssymb,amsmath,graphicx}
\usepackage{graphics}
\usepackage[english]{babel}
\newtheorem{thm}{Theorem}

\newtheorem{ex}{Example}
\newtheorem{exs}{Exercise}

\newtheorem{rk}{Remark}

\theoremstyle{definition}
\newtheorem{dfn}{Definition}
\newtheorem{notat}{Notation}

\newcommand\Na{\mbox{Set}}

\title{Reidemeister Moves and Groups}
\author{Vassily Olegovich Manturov\footnote{The author is partially
supported by Laboratory of Quantum Topology of Chelyabinsk State
University (Russian Federation government grant 14.Z50.31.0020), by
RF President NSh Ч 1410.2012.1,
 and
by grants of the Russian Foundation for Basic Resarch,
13-01-00830,14-01-91161, 14-01-31288.}}

\date{}

\def\R{{\mathbb R}}

 \def\Z{{\mathbb Z}}
 \def\0{{\mathbbf 0}}
 \def\1{{\mathbbf 1}}

 \def\gG{{\mathfrak G}}

 \newcommand{\skcr}{\raisebox{-0.25\height}{\includegraphics[width=0.5cm]{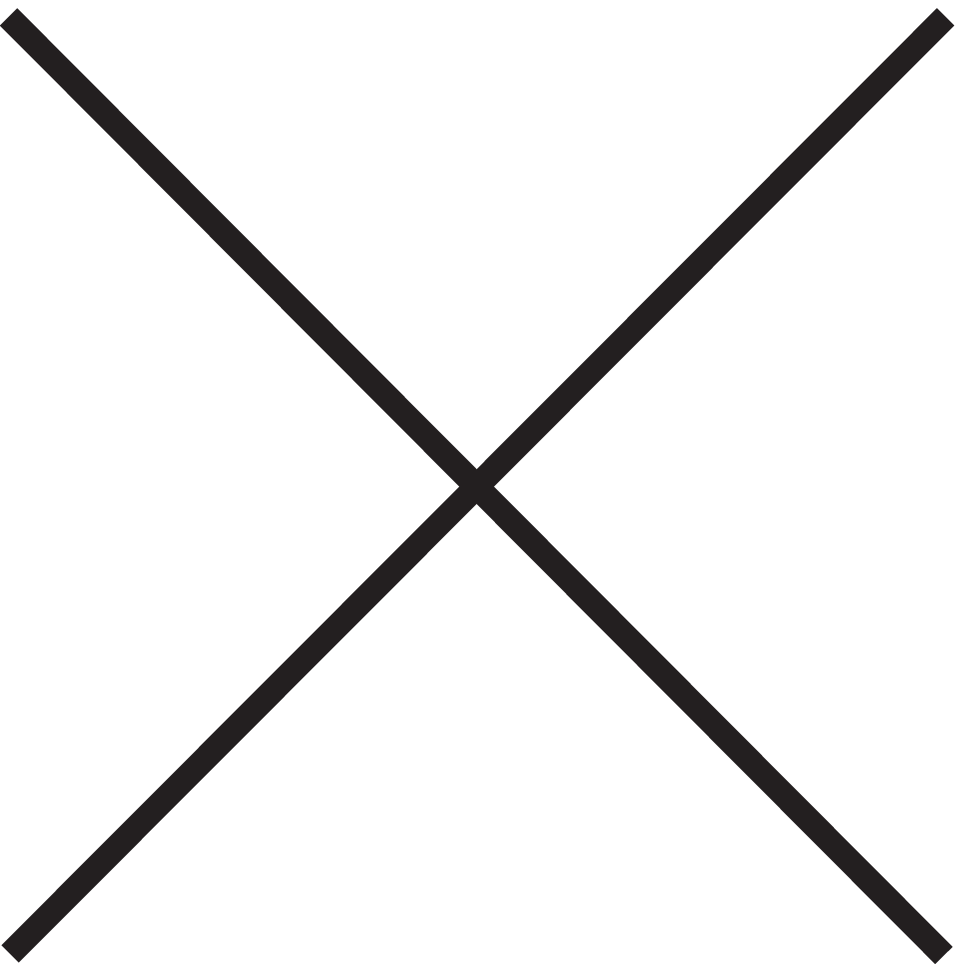}}}
 \newcommand{\skcrv}{\raisebox{-0.25\height}{\includegraphics[width=0.5cm]{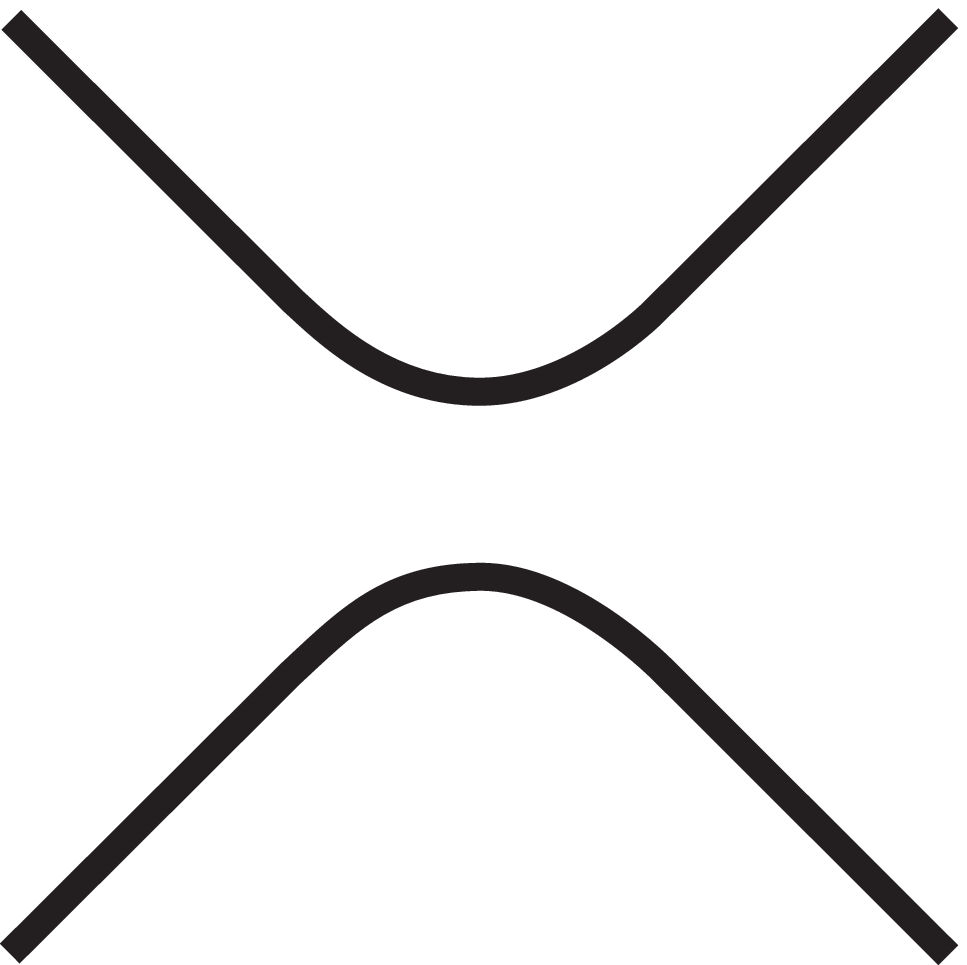}}}
\newcommand{\skcrh}{\raisebox{-0.25\height}{\includegraphics[width=0.5cm]{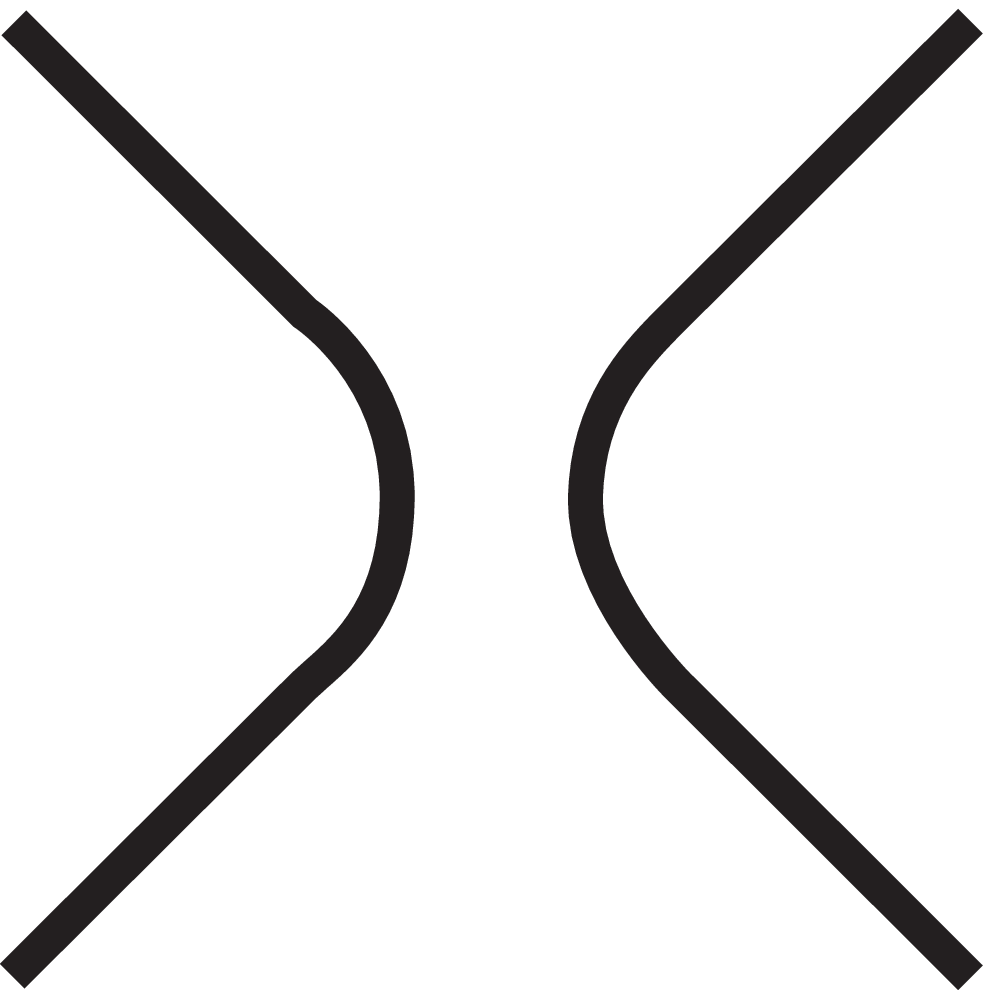}}}
\newcommand{\skcrosso}{\raisebox{-0.25\height}{\includegraphics[width=0.5cm]{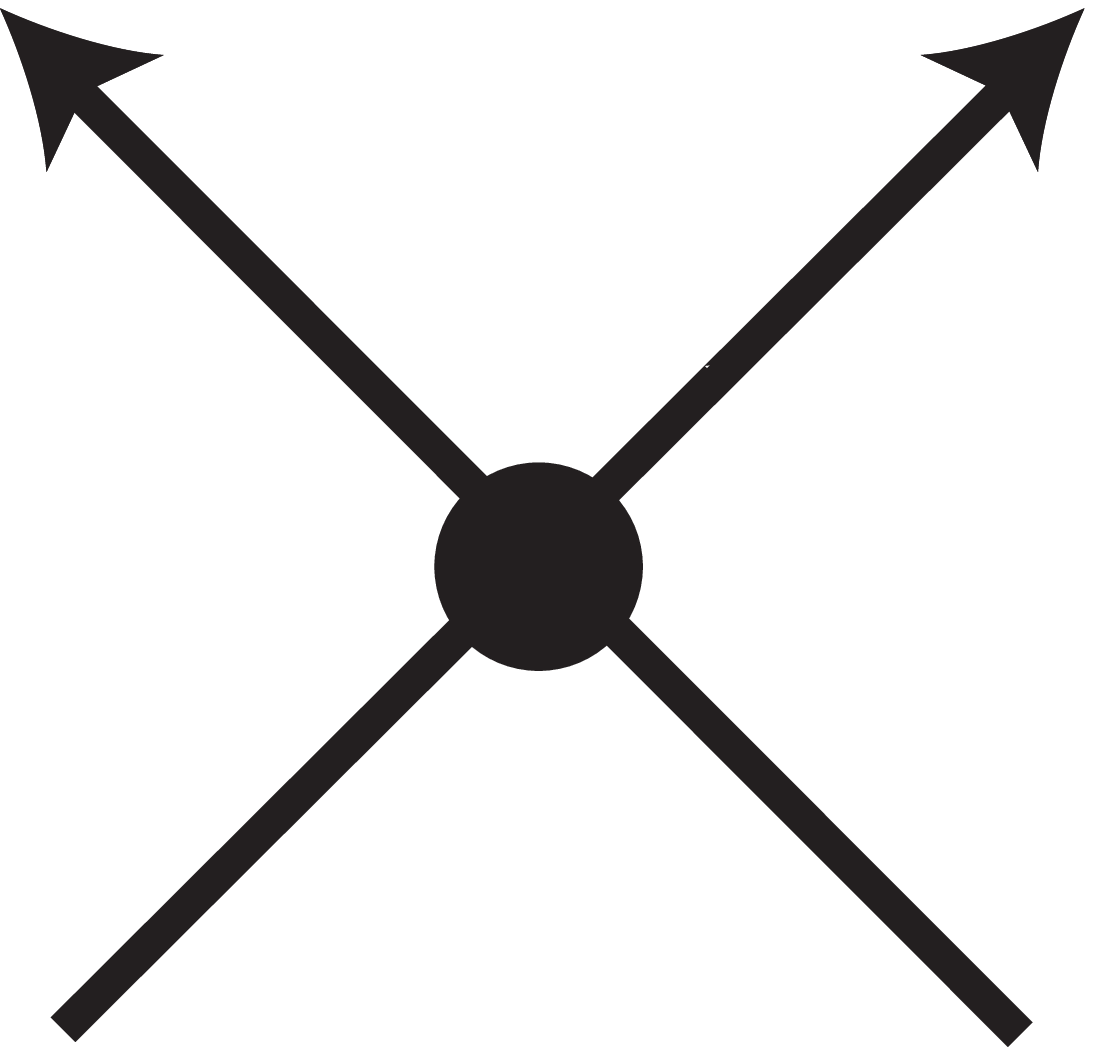}}}

 \newcommand{\skcrho}{\raisebox{-0.25\height}{\includegraphics[width=0.5cm]{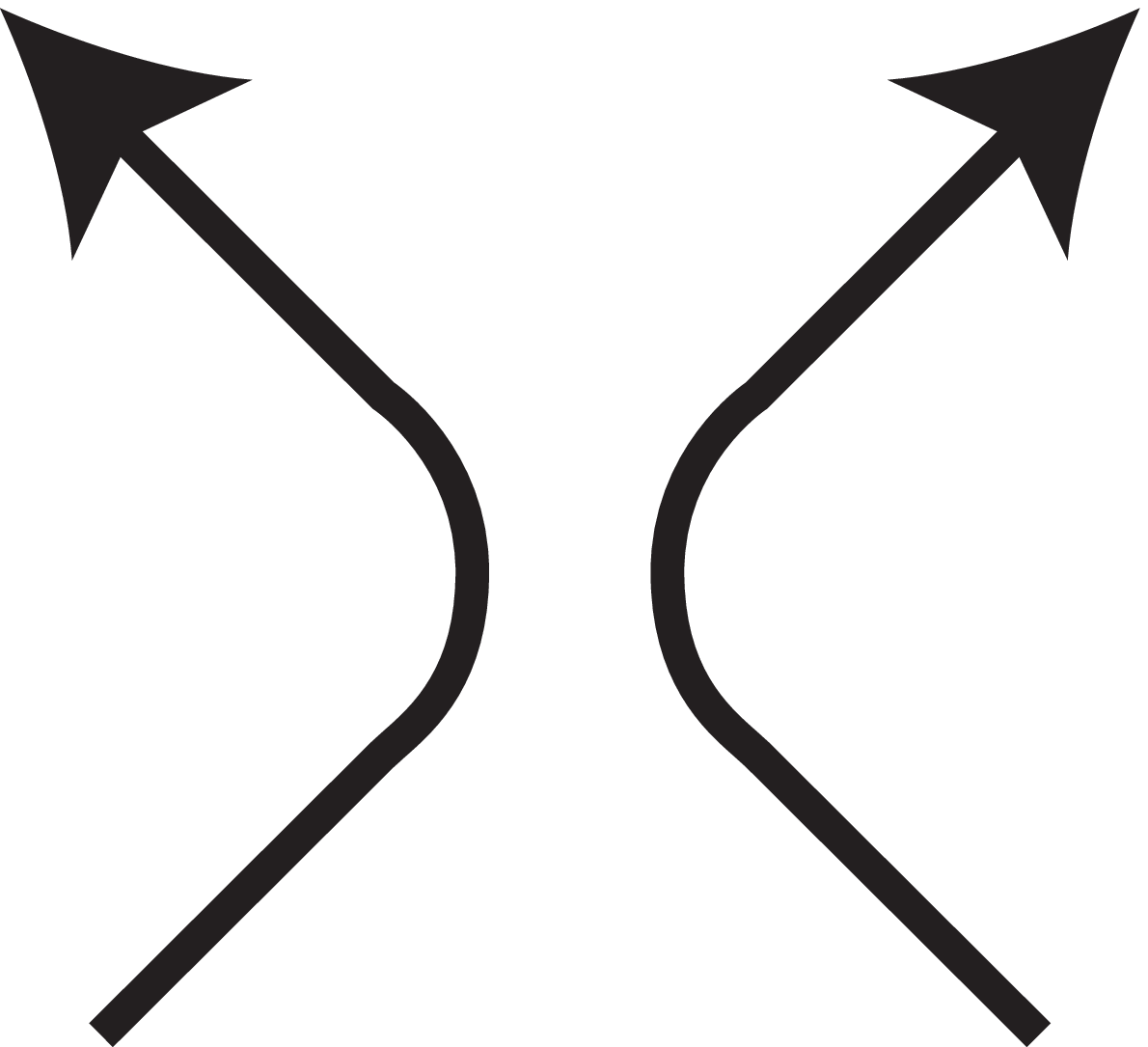}}}

\begin{document}

\maketitle

\begin{abstract}Recently, the author discovered an interesting class of
knot-like objects called {\em free knots}. These purely
combinatorial objects are equivalence classes of Gauss diagrams
modulo Reidemeister moves (the same notion in the language of words
was introduced by Turaev \cite{Turaev}, who thought all free knots
to be trivial). As it turned out, these new objects are highly
non-trivial, see \cite{Parity}, and even admit non-trivial cobordism
classes \cite{Cobordisms}. An important issue is the existence of
invariants where a {\em diagram evaluates to itself} which makes
such objects ``similar'' to free groups: an element has its minimal
representative which ``lives inside'' any representative equivalent
to it.

In the present paper, we consider generalizations of free knots by
means of (finitely presented) groups. These new objects have lots of
non-trivial properties coming from both knot theory and group
theory: functoriality, coverings, etc. This connection allows one
not only to apply group theory to various problems in knot theory
but also to apply Reidemeister moves to the study of (finitely
presented) groups.

Groups appear naturally in this setting when graphs are embedded in
$2$-surfaces.

\end{abstract}

\section{Introduction. Basic Definitions and Notation}

Knots and virtual knots are encoded by $4$-valent graphs $\Gamma$
with a {\em framing}  (opposite edge structure at vertices) and some
decorations at vertices (the structure of overpasses and
underpasses) modulo {\em Reidemeister moves}. When skipping
de\-co\-rations of crossings, we get to certain equivalence classes
of virtual knots, called {\em free knots}. If such a graph $\Gamma$
is treated as the image of immersion of a curve $\gamma$ in an
oriented $2$-surface $\Sigma$ in such a way that $\Gamma\backslash
\gamma$ admits a checkerboard colouring, then there is a natural
presentation of the quotient group $\pi_{1}(\Sigma)\slash \langle
[\gamma]\rangle$ by the normal closure of the element corresponding
to $\gamma$ where {\em vertices are generators} and {\em regions are
relators}; for more details, see \cite{IMN,FedoseevManturov}.

Thus, considering $\Gamma$ as a  homotopy class of $\gamma$,
vertices of $\Gamma$ get some natural interpretations in terms of
homotopy classes. When performing Reidemeister moves to $\Gamma$,
labels of vertices undergo natural transformations.

When considering more abstract graphs unrelated to any surfaces, one
gets the notion of a {\em free knot} (in the case of many components
$\gamma_{i}$, one can talk about a {\em free link}). A {\em free
knot} is a $1$-component free link.

Free knots possess highly non-trivial invariants demonstrating the
following principle:

{\em if a free link diagram $K$ is complicated enough, then every
diagram $K'$ equivalent to $K$ contains $K$ inside}.

This result is achieved by using the parity bracket $[\cdot]$, an
invariant of free links valued in {\em diagrams of free links} such
that $[K]=K$ for diagrams which are large enough, see ahead.
 This allowed the author to get new
approaches to various problems in virtual \cite{Projection,Crossing}
and classical knot theory \cite{ChrismanManturov,KrasnovManturov},
and various questions of topology.

Thus, free links are interesting objects by themselves; on the other
hand, vertices of their diagrams (framed $4$-graphs) can be labeled
by elements of certain groups.

This leads to the notion of {\em group free knots} or {\em $G$-free
knots} where $G$ is a (finitely presented) group. When forgetting
labelings of vertices of framed $4$-graphs, we get usual
(classical,virtual, flat or free) knot theory, when embedding a
framed $4$-graph into an oriented $2$-surface, we get a natural
labeling.

When applying Reidemeister moves to diagrams within a $2$-surface,
we naturally get moves for labeled diagrams.

Thus, considering various questions concerning equivalence classes
of group free knots, we can study arbitrary groups. Indeed, every
group $G$ leads to its own $G$-labeled free knot theory, and
depending on which graphs turn out to be equivalent in this group,
one may judge about the structure of the group.

Let us now pass to formal definitions of the objects, we are going
to deal with.

\begin{rk}
Throughout the rest of the text, all groups are assumed to be
finitely presented.
\end{rk}

We shall deal with framed $4$-graphs which naturally appear as
shadows of knot diagrams.

\begin{rk}
We say {\em ``$4$-graph''} instead of {\em ``$4$-valent graph''},
because we also admit objects to have {\em circular components},
i.e., circles which are disjoint from the graph.
\end{rk}

\begin{dfn}
A $4$-valent graph $\Gamma$ is {\em framed} if at every vertex $X$,
the four half-edges incident to it, are split into two sets;
(half)edges from the same set are called {\em opposite}; (half)edges
which are not opposite are called {\em adjacent}; such a choice is
called {\em framing} of $\Gamma$.
\end{dfn}

\begin{dfn}
By a {\em rotating cycle} of a framed $4$-graph $\Gamma$ we mean
either a circular component of $\Gamma$ or a sequence of vertices
(possibly, non-distinct) $v_{i}$ and distinct edges $e_{i},
i=0,\dots, k-1,$ such that $e_{i}$ connects $v_{i}$ and $v_{i+1}$
and at $v_{i},$ the edges $e_{i}$ and $e_{i-1}$ are non opposite.
Here indices are counted modulo $k$.
\end{dfn}

\begin{dfn}
When talking about the number of components, we separately count
{\em circular components} and separately count {\em unicursal
components}; by the latter, we mean an equivalence class of edges
called by elementary equivalence relation where every two edges
opposite at some vertex are treated as opposite.
\end{dfn}

Thus, the  disjoint diagram of the disjoint sum of the simplest
$2$-vertex Hopf link and the unknot has three components: a circular
one corresponding to the unknot and the two unicursal components of
the Hopf link.

\begin{dfn}
A {\em source-sink structure} of a framed $4$-graph $\Gamma$ is an
orientation of all edges of $\Gamma$ such that at every vertex $X$
of it some two opposite (half)edges are emanating and the other two
are incoming, see Fig. \ref{SourceSink}.

Circular edges are just assumed to be oriented.
\end{dfn}

\begin{figure}
\centering\includegraphics[width=200pt]{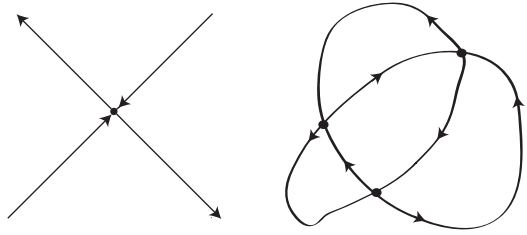} \caption{A
Source-Sink Structure at a Vertex; A Graph with a Source-Sink
Structure} \label{SourceSink}
\end{figure}

\begin{dfn}
Let $G$ be a framed $4$-graph, and let $X$ be a vertex of $G$, let
$a,b$ be one pair of opposite (half)edges, and let $c,d$ be the the
other pair of opposite (half)edges at $X$; by the {\em vertex
removal} we mean the operation of deleting the vertex $X$ from $G$
and connecting $a$ to $b$ and connecting $c$ to $d$.
\end{dfn}

\begin{rk}
Note that if a connected framed $4$-graph admits a source-sink
structure then there exist exactly two such structures which differ
by the overall orientation change.
\end{rk}

\begin{notat}
We adopt the following convention. Whenever we draw an immersion of
a framed $4$-graph in $\R^{2}$, we depict its vertices by solid
circles, those points which are encircled are artifacts of
projection caused by immersion. At every vertex, the framing is
assumed to be {\em induced from the plane}: (half)edges which are
locally opposite on the plane are opposite.
\end{notat}

Note that if at a vertex $X$ of a framed $4$-graph a half-edge $a$
is opposite to a half-edge $c$ and a half-edge $b$ is opposite to a
half-edge $d$, then then when drawing on the plane, the
counterclockwise order of edges can be $a,b,c,d$ or $a,d,c,b$.

 This leads to the two ``moves'' for planar diagrams of
$4$-graphs which do not change the framed $4$-graph: the {\em detour
move} and the {\em virtualization move}. The {\em detour move}
removes one piece of an edge with all encircled crossings inside it
and redraws it elsewhere with new encircled crossings with itself
and other edges, see Fig. \ref{detour}; the {\em virtualiation move}
changes the local counterclockwise order at a vertex without
changing its framing; one can represent this move by flanking the
classical crossings by two encircled intersection points of $a,b$
and $c,d$.

\begin{figure}
\centering\includegraphics[width=150pt]{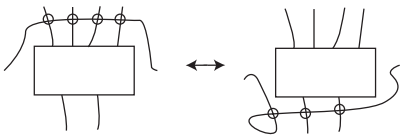} \caption{The
detour move} \label{detour}
\end{figure}

\begin{rk}
In the sequel, when drawing some transformation on the plane, we
shall depict only the changing part of the move assuming the
remaining part to be fixed.
\end{rk}

\begin{ex}
Consider the framed $4$-graph with one vertex $A$ and two loops, $p$
and $q$ such that $p$ is opposite to itself at $A$ and $q$ is
opposite to itself at $A$. Then this graph admits no source-sink
structure because whatever orientation of the edge $p$ we take, it
will be emanating from one side and incoming from the other side.
\end{ex}

\begin{dfn}
We say that a framed $4$-graph is {\em good} if it admits a
source-sink structure.

We say that a good framed $4$-graphs {\em oriented} if a source-sink
structure of it is selected.

Likewise, we  define {\em good} (resp., {\em oriented}) $G$-framed
$4$ graphs $(\Gamma, f)$ if $\Gamma$ is good (resp., oriented).
\end{dfn}

\begin{rk}
We shall refer to {\em oriented} $G$-framed $4$-graphs simply as
$G$-graphs. Moreover, we shall often omit $f$ from the notation if
the definition of $f$ is clear from the context.
\end{rk}

\begin{dfn}
Let $G$ be a group. An {\em oriented} $G$-free link is an
equivalence class of $G$-graphs by the following three Reidemeister
moves shown in Fig.\ref{rmoves}.

\begin{figure}
\centering\includegraphics[width=200pt]{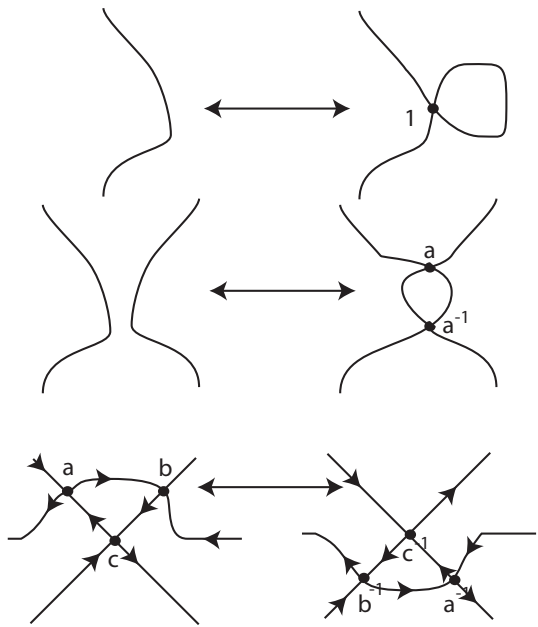}
\caption{The Reidemeister moves for free links}\label{rmoves}
\end{figure}
\end{dfn}

Here we assume that for the first Reidemeister move, the vertex
taking part in this move is marked by the unit element of the group,
for the second Reidemeister moves, the two elements are inverse to
each other, and for the third move we have some three elements
$a,b,c$ in the left hand side such that $a\cdot b\cdot c=1$, and the
opposite elements $a^{-1},b^{-1},c^{-1}$ in the right hand side, see
Fig. \ref{rmoves}.

Note that for the first Reidemeister move the choice of the
source-sink orientation does not matter; it only matters whether
these orientations agree for the edges touching the boundary of the
picture; for the third Reidemeister move, it is crucial to require
that the cyclic order of the three vertices along the source-sink
orientation is $a,b,c$ and not $a,c,b$.

\begin{dfn}
Let $G$ be a finitely presented group. By a {\em
$G$-framed-$4$-graph} we mean a pair $(\Gamma,f)$ where $\Gamma$ is
a framed $4$-graph and $f$ is the map from the set of vertices of
$\Gamma$ to $G$.
\end{dfn}

Let $\Gamma$ be a $G$-graph. Let $\gamma$ be a rotating cycle on
$\Gamma$ with $k$ vertices (a vertex is counted twice if it is
passed twice). Then $\gamma$ inherits a natural orientation from the
source-sink structure of $\Gamma$. If we choose a starting point on
$\gamma$, we can write down the elements of $G$ corresponding to the
vertices of the group $G:$ $[\gamma]=\gamma_{1},\cdots,\gamma_{k}\in
G$. If no starting point is given, we can obtain a conjugacy class
of $\gamma$.

\begin{dfn}
The {\em minimal crossing number} of a $G$-free link $L$ is the
minimal number of vertices of oriented $G$-framed $4$-graphs
representing $L$.
\end{dfn}

This definition means that every group $G$ leads to a well defined
{\em free knot theory} corresponding to this group. Even for the
trivial group, this theory is extremely interesting: it gives rise
to the theory of {\em even} free knots (i.e. free knots whose
diagrams admit source-sink structure), which has non-trivial
invariants reproducing itself \cite{Parity,Doklady}.

Thus, one can use this theory in both directions: it is possible to
study groups by means of free knots and their generalizations and to
study various knot theories by using group-valued labeling of their
vertices.

\begin{rk}
A similar theory can be constructed for all free knots, not
necessary having source-sink structure, however, in this way one
should overcome some difficulties with orientations corresponding to
relations for Reidemeister moves. We shall touch on this question in
another paper.

Note also that, in the case of the $\Z_{2}$-group, the theory can be
constructed even in the case of arbitrary graphs (see ahead).
\end{rk}

The paper is organized as follows. In the next section, we present
the three main tools for constructing invariants of our object, {\em
the parity bracket}, {\em the parity delta}, and {\em the covering}.

In the third section, we shall consider various examples where the
group labeling  arises in topological situation, and which relations
on groups does this group labeling are imposed in concrete
situation. We shall mostly concentrate on the case of group coming
from a $2$-surface.

We conclude the paper by a list of unsolved problems relating
Reidemeister moves with group theory.

\section{Basic Invariants}

We are now ready to construct our first invariants of groups. Let
$G$ be a finite group.

\begin{dfn}
Let ${\cal R}(G)$ be the $\Z_{2}$-module freely generated by all
$G$-knots.
\end{dfn}

\begin{dfn}
Let $G$ be a group. By $\Na(G,k)$ we mean the formal integral linear
combination of all $G$-links $K$ having minimal crossing number $k$.
\end{dfn}

From definition, it follows that $\Na(G,k)$ is an invariant of the
group $G$.

\begin{rk}
For the case of infinite $G$, the number of all $G$-links $K$ having
minimal crossing number $k$ can be infinite.
\end{rk}

\begin{ex}
Let $G$ be the trivial group. Then the $G$-link theory coincides
with the theory of (having source-sink structure and oriented) free
links.
\end{ex}

One of the main phenomena in virtual knot theory is the {\em local
crossing information}. Having two (or many) different types of
crossings which behave nicely under Reidemeister moves, we can
enhance many invariants by localizing some information at crossings;
moreover, this allows one to reduce the study of objects to the
study of their diagrams.

In the case of the $\Z_{2}$-group, the main invariants are called
{\em the parity bracket}, {\em the parity projection}, and {\em
parity covering}. We are going to describe the parity bracket and to
generalize the projection for the case of arbitrary groups.

The {\em group bracket}, unlike the parity bracket, does not reduce
the study of knots to the study of their diagrams, but rather is a
functorial mapping from the category of $G$-knots to $G$-knots.

Actually, iterative use of projections, and coverings and other
tools lead to enhancement of many invariants constructed
combinatorially.

In a similar way, one can enhance various combinatorial invariants
in the case of arbitrary group $G$. We shall touch on these problems
in subsequent papers.

Another important issue in the virtual knot theory is that many
invariants (like linking numbers or writhe numbers) become
picture-valued: instead of some count of crossings (with signs), we
are allowed to count pictures arising at these points.

\subsection{The Parity Bracket}

In the present section, we will concentrate on the case of the group
$\Z_{2}$.

We will define the {\em parity bracket} which is well-defined for
$\Z_{2}$-graphs and which realizes the principle ``if a knot (link)
diagram is complicated enough then it reproduces itself''.

In this case, {\em complicated enough} will mean {\em irreducible
and odd}, see ahead, however, for other groups $G$ there are many
other situation where it works.

A $\Z_{2}$-labeling $f$ of a $G$-graph $\Gamma$ is usually referred
to as {\em parity}.

This bracket was first defined in \cite{Parity} for the so-called
{\em Gaussian parity} which is trivial for all good framed
$4$-graphs. Nevertheless, all the results stated below (as well as
their proofs) remain true {\em whatever parity we take.}

\begin{dfn}
By a {\em smoothing} of a framed $4$-graph at a vertex we mean the
result of deleting this vertex and repasting the half-edges into two
edges in one of the two possible ways $\skcr\to \skcrv$,$\skcr\to
\skcrh$.
\end{dfn}

\begin{rk}
Note that the smoothing may result in new circular components of a
free graph.
 Note that a smoothing of a {\em good} framed
$4$-graph is {\em good}.
\end{rk}

Denote by $\gG$ the space of  $\mathbb{Z}_{2}$-linear combinations
of the following objects. We consider the linear space of all framed
$4$-graphs (not necessarily good ones) subject to the second
Reidemeister moves.

By a {\em bigon} of a framed $4$-graph $\Gamma$ we mean two edges
$e_{1}, e_{2}$ which connect some two vertices $v_{1}$ and $v_{2}$
and are non-opposite at both these vertices; bigons appear in the
second Reidemeister moves. A framed $4$-graph is {\em irreducible}
if it has no bigons inside.

\begin{rk}
It can be easily seen that the equivalence classes modulo the second
Reidemeister moves are characterized by its minimal representatives,
i.e., framed $4$-graphs without ``bigons''. Thus, we shall use the
term ``graph'' when talking about an element from $\gG$ assuming the
minimal representative of its element.
\end{rk}

Let $K$ be a framed good $Z_{2}$ $4$-graph. The bracket invariant
$[\cdot]$ (see \cite{Parity}) is given by
 $$
[K]=\sum_{s_{{even,1-component}}}K_{s}\in \gG,
 $$
where the sum is taken by all smoothings $s$ at even vertices which
lead to one-component diagrams $K_{s}$.
 This sum is considered as an
element from $\gG$.

The following theorem is proved in \cite{Parity}
\begin{thm}
$[\cdot]$ is an invariant of free knots; in other words, if two
framed $4$-graphs $\Gamma_{1}$ and $\Gamma_{1}$ are equivalent then
$[\Gamma_{1}]=[\Gamma_{2}]\in \gG$.
\end{thm}

Let us call a  $\Z_{2}$-graph {\em odd} if all vertices of it are
odd.

It follows from the definition of the bracket that if $\Gamma$ is
odd and irreducible then $[\Gamma]=[\Gamma]$. In particular, this
means that {\em for every irreducible and odd $\Z_{2}$-graph
$\Gamma$ and every graph $\Gamma'$ equivalent to it, $\Gamma$ can be
obtained from $\Gamma'$ by means of smoothings at some vertices.}

\subsection{Turaev's Delta and its Generalizations}

Actually, many integer-valued invariants of topological objects can
be calculated as certain (algebraic) sums taken over certain
reference points (e.g., some intersection points in some
configuration spaces etc). It turns out that in many situations
which happen in low-dimensional topology, these points can be
themselves endowed with certain (topological or cominatorial)
information which makes them responsible for non-triviality of the
object itself and various properties of it. Moreover, in some cases
these points can be associated with objects similar to the initial
one.

This switch from numbers to pictures changes the situation
crucially. The (algebraic) sums instead of being just integer-valued
invariants of the initial object, transform into functorial mappings
from objects to similar objects.

For curves in $2$-surfaces, there are two operations, the
multiplication and the comultiplication. The first one is due to
W.Goldman \cite{Goldman}, and the second one is due to V.G.Turaev
\cite{Turaev2}. The main idea is that we can get knots from
$2$-component links and $2$-component links from knots by taking
smoothings at some crossings. Then we take sums over such crossings
and get an invariant map. We are mostly interested in the second
operation and call it and its generalization ``Turaev's delta''.

However, if we forget about $2$-surfaces and deal with curves in
general position, we can take care not about their homotopy classes,
but rather about free knots or links which appear at these crossings
(see \cite{IMN} for free knots). But if we deal not with bare free
knots, we can enhance the definition and take into account some
group information.

Let us be more specific.

Let $\gamma$ be a one-component framed $4$-graph.

At each crossing $c$, 
onecan perform the following {\em oriented smoothing} $\skcrosso\to
\skcrho$. Let us denote the result of such smoothing at $c$ by
$\gamma_{c}$.

Then $\gamma_{c}$ is a framed $4$-graph with two unicursal
components.

Consider the  linear space ${\cal M}$ generated by free
two-component links.

Now let ${\cal L}$ be the quotient of links in ${\cal M}$ by those
links having one trivial component.

Let $\Delta(\gamma)=\sum_{c} \gamma_{c} .$

\begin{thm}
$\Delta$ is a well-defined map from the set of free knots
 to ${\cal
L}$.
\end{thm}

The proof follows from a consideration of all Reidemeister moves.
Looking carefully at Reidemeister moves, we see that when taking sum
in ${\cal M}$, the invariance Reidemeister move requires the
two-component free link with one trivial component to be zero.

\begin{rk}
For the case of $2$-surfaces, we have homotopy classes of curves
which are represented by curves in general position; the latter have
intersection points which correspond to vertices of the framed
$4$-graph.
\end{rk}

If we had not curves in $2$-surfaces but $G$-knots it would be
sufficient to define our $\Delta_{G}$ to be

$$\Delta_{G}=\sum_{c, l(c)\neq 1}\gamma_{c},$$
to be the sum over all crossings $c$ with non-trivial label $l(c)$
of the resulting two-component free links.

Note that the invariance under the third Reidemeister move does not
allow one to consider the summands as $G$-links.

However, one can split $\Delta_{G}$ with respect to elements of
$\Delta$. For every element $g\neq 1\in G$ we define

$\Delta_{g}=\Delta_{g}=\Delta_{g^{-1}}$ to be

$$\sum_{c,l(c)\in \{g,g^{-1}\}} \gamma_{c}.$$

Then all $\Delta_{g}$ are invariants of the initial link and
$\Delta_{G}$ naturally splits into the sum over all classes of
non-trivial inverse elements.

As we shall see later, the structure of $G$-knots naturally appears
for curves in $2$-surfaces.

\begin{thm}
For each $g$, the mapping $\Delta_{g}$ is invariant.

Consequently, $\Delta_{G}$ is an invariant mapping.
\end{thm}

 \subsection{The group bracket}

    As we have seen before, $\Z_{2}$-knots admit a natural bracket
$[\cdot]$ which takes all $\Z_{2}$-knots into $\Z_{2}$-linear
combinations of equivalence classes of free knots modulo second
Reidemeister moves. This actually means that we get from knots to
graphs.

   This actually happens because there exist exactly one non-trivial
element in $\Z_{2}$.

    For arbitrary group $G$, there is a similar bracket operation which
takes $G$-knots to equivalence classes of diagrams modulo moves;
however, the target space will consist of equivalence classes of
$G$-graphs modulo not only second Reidemeister moves but also third
Reidemeister moves. Thus, the group bracket unlike the partiy
bracket, can not be treated as a graph-valued invariant, but rather,
as a functorial map for $G$-knots, similar to Turaev's delta.

      Let us be more specific.

      Let ${\cal S}_{G}$ be the set of $\Z_{2}$-equivalence classes of
$G$-graphs with no vertex marked by the unit element of the group,
modulo the second and the third Reidemeister moves.

       Let $K$ be a $G$-graph. Consider the set of all vertices of $K$
marked by the unit element of the group $G$. We define smoothings
$s$ at unit vertices of $K$ as before. Now, we set

 $$
[K]_{G}=\sum_{s_{{even,1-component}}}K_{s}\in {\cal S}_{G}.
 $$

\begin{thm}
The map ${\cal S}_{G}$ is an invariant of $G$-knots.
\end{thm}

\begin{proof}
We have to check the invariance under the three Reidemeister moves.

Let $K$ and $K'$ be two $G$-graphs which differ by a Reidemeister
move at some vertex. We shall show that $[K]=[K']$.

Assume first $K'$ is obtained from $K$ by the first Reidemeister
move which adds one vertex $v$ with trivial label. Then $[K']$
contains summands where $v$ yields a local loop. They do not count
in the bracket since they have at least two components.

All other summands in $[K']$ are in one-to-one correspondence with
all summands in $[K]$.

Assume now $K'$ is obtained from $K$ by a second Reidemeister move
where the two additional crossings in $K'$ are both non-trivial.

Then all summands in $[K']$ are in one-to-one correspondence with
those of $[K]$, and they differ exactly by the same second
Reidemeister move.

The same situation happens when $K'$ differs from $K$ by a third
Reidemeister move with three vertices of non-trivial labels: we have
a one-to-one correspondence of equal summands.

The cases when $K'$ differs from $K$ by a second Reidemeister move
with two additional trivial crossings or by a third Reidemeister
move with two additional trivial crossings are actually the same as
in the case of usual bracket; for more detail, see \cite{Parity}.

Finally, assume $K'$ differs from $K$ by a third Reidemeister move
with labels $a,b,c$ such that $b=a^{-1}$ and $c=1$ (see Fig.
\ref{rmoves}). In the right hand side, we will have labels
$a^{-1},b^{-1}=a$, and  $1^{-1}=1$.

Now we see that one  smoothing at the vertex of $K$ labeled $1$
coincides identically with the similar smoothing of $K$; the other
smoothing of $K$ differs by the similar smoothing of $K'$ by two
second Reidemeister moves.

\end{proof}

\subsection{Coverings and Projections}

First, let $\alpha:G\to H$ be a homomorphism of two groups. Then,
for every $G$-graph $\Gamma$, we get an $H$-graph
$\alpha_{*}(\Gamma)$ by changing each vertex label $g\in G$ to
$\alpha(g)\in H$.

From the definition of $G$-knots we get to the following
\begin{thm}
The operation $\alpha_{*}$ leads to a well defined mapping from
$G$-knots to $H$-knots, i.e., if $\Gamma_{1},\Gamma_{2}$ are two
equivalent $G$-graphs then their images $\alpha_{*}(\Gamma_{1})$
$\alpha_{*}(\Gamma_{2})$ are two equivalent $H$-graphs.
\end{thm}

Below, we briefly sketch the idea how to construct a covering for
$G$-links for some subgroup $G'$ of a group $G$; we define the
corresponding projection in the general case; as for the covering,
we define it only in the abelian case; in the general case it is
defined quite analogously (see also subsection \ref{subscurves}
curves in $2$-surfaces).

Now, let $G'$ be a subgroup of a group $G$. Then for every $G$-graph
$\Gamma$, we can define the $G'$-graph $\beta_{G'}(\Gamma)$ obtained
from $\alpha_{*}$ by removing all vertices of $\Gamma$ whose labels
do not belong to $G'$; the remaining vertices will keep their labels
considered now as elements of $G'$.

\begin{thm}
The operation $\beta_{G'}(\Gamma)$ is a well defined mapping from
$G$-knots to $G'$-knots.
\end{thm}

Thus, we have constructed a map $\beta_{G'}$ which {\em removes}
some nodes. In particular, this projection maps knots to knots.

In the simplest case when $G=\Z_{2}, G'=\{1\}$, the {\em parity
projection} map can be justified to a parity covering map. Namely,
the projection map takes knots to $2$-component links, and
generally, $l$-component links to $2l$-component links with two sets
of $l$ components each. Over each crossings of the initial knot $X$,
there are two crossings $X_{1}$ and $X_{2}$, of the obtained link.
 The  {\em
odd} crossings (those marked by non-trivial elements of $\Z_{2}$
which disappear under the projection map), will belong to two
different components (sets of components) \cite{NewParity} (``mixed
crossings''). Thus, over each crossing marked by the trivial
element, there will be two {\em pure} crossings belonging to one
component (in the case of links, to one set of components),
crossings marked by the non-trivial element will be covered by two
crossings, each being mixed (i.e., belonging to two different sets
of components).

Now, let $G$ be an abelian group. Let $G'$ be a subgroup\footnote{In
the case of general $G$, one should consider normal subgroups only}
of $G$ and let $H$ be the quotient group so that we have a short
exact sequence

$$G'\stackrel{\beta}{\to} G\stackrel{\alpha}{\to} H.$$

Let $(L,f)$ be a $G$-graph with $k$ components; we denote the
$G$-link represented by $L,k$ by the same letter $L$. We shall
construct the {\em the covering} corresponding to the inclusion
$\beta:G\to G'$, as follows.

This covering will be a $H$-link having $k\cdot |H|$ components.
This link will be denoted by $L^{\beta}$.

With each vertex $X$ of $(L,f)$ with label $g\in G,$ where
$g=g'h,g'\in G',h\in H$, we  shall associate $H$ vertices $X^{h},
h=\alpha(g)\in H$ indexed by elements of $H$; all these vertices
will have label $g'$.

Now, with every edge $XY$ connecting some two vertices $X,Y$ of $L$
we associate $|H|$ edges, namely, $X^{g}$ is connected to $Y^{gh}$
where $h=\alpha(g)$.

\begin{thm}
This map $L\to L^{\beta}$ is a well defined map from $G$-free links
to $G'$-free links.
\end{thm}

\begin{proof}
This theorem follows from the consideration of Reidemeister moves:
when looking at the product of labels along the loop, bigon, or
triangle, we see that these products are all equal to the unit of
the group $G$, which means that both ``the $G'$-factor'' and ``the
$H$-factor'' of them are units of $G'$ and $H$, respectively. This
means that the corresponding edges of $L^{\beta}$ will close up to a
loop, bigon or triangle, repectively. Hence, the Reidemeister move
applied to $L$ will result in the corresponding Reidemeister move
applied to $L^{\beta}$.
\end{proof}

\section{Further Applications and Examples}

\subsection{Minimal Crossing Number}

Consider the example shown in Fig. \ref{F1}.

\begin{figure}
\centering\includegraphics[width=200pt]{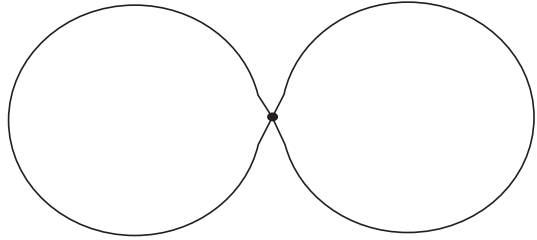} \caption{The graph
$F_{1}$}\label{F1}
\end{figure}

For every group $G$, the framed graph $F_{1}$ is equivalent to the
unknot if the label on the unique vertex is equal to the unit of the
group $G$ as a $G$-knot.

\begin{thm}
For every non-trivial group $G$, there exists a non-trivial labeling
$f$  of the graph $F_{1}$ such that the $G$-knot $(F_{1},f)$ is not
equivalent to the unknot.
\end{thm}

Indeed, $\Delta(F_{1})$ has one non-trivial summand.

Now, let us consider the example shown in Fig. \ref{F2}.

\begin{figure}
\centering\includegraphics[width=200pt]{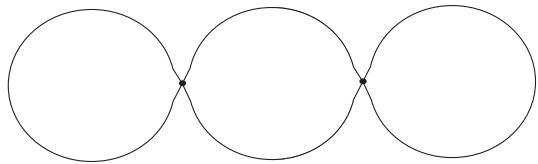} \caption{The graph
$F_{2}$} \label{F2}
\end{figure}

As well as $F_{1}$, the framed $4$-graph $F_{2}$ is trivial for the
$1$-element group. Moreover, if we take $\Z_{2}$, then for every
labeling of $F_{2}$,

Thus, $F_{2}$ is not minimal for $\Z_{2}$.

\begin{thm}
For every group $G$ containing at least three elements, there exists
a labeling $f$ of vertices of $F_{2}$ which makes this labeled graph
minimal in its $G$-knot class.
\end{thm}
\subsection{Curves in $2$-surfaces}

\label{subscurves}

In the present section, we give the main example where the above
theory comes from.

Let $\Sigma$ be a closed oriented $2$-surface of genus $g$. We say
that an embedding $f$ of a framed $4$-graph $\Gamma$ in $\Sigma$ is
cellular if $\Sigma\backslash \Gamma$ is the disjoint union of
$2$-cells.

We say that a cellular embedding $f:\Gamma \to \Sigma$ is {\em
checkerboard} if one can colour all cells of $\Sigma$ with $2$
colours in a way such that no two cells of the same colour are
adjacent along an edge.

The following fact can be proved by the reader as a simple exercise.
\begin{exs}
If a framed $4$-graph $\Gamma$ admits a checkerboard cellular
embedding into some $\Sigma$ then $\Gamma$ admits a source-sink
structure.

Moreover, if $\Gamma$ admits a source-sink structure then every
cellular embedding of $\Gamma$ in every $\Sigma$ (of any genus) is
checkerboard.
\end{exs}

Now, let us fix a framed graph $\Gamma$ and its checkerboard
cellular embedding $f$ into $\Sigma=S_{g}$. Consider the group
$G(\Gamma,f)$ given by the following presentation $P$.

The generators of $P$ are in one-to-one correspondence with vertices
of $\Gamma$.

The relaions of $P$ are in one-to-one-correspondence with cells of
$\Sigma\backslash f(\Gamma)$. More precisely, with each cell $C$,
one
can associate its boundary which is a {\em rotating cycle} 
on $\Gamma$. Since $\Gamma$ admits a source-sink structure, this
cycle becomes naturally oriented. This means, that all vertices of
this cycle are cyclically ordered. This generates a cyclic word
which we take to be the relator corresponding to the cell.

Now, assume $\Gamma$ has one unicursal component. Then $f(\Gamma)$
can be thought of as the image of an immersed curve; we shall call
this image $\gamma$.

In \cite{Nikonov}, I.M.Nikonov proved the following
\begin{thm}
If $\Gamma$ has one unicursal component, then the group
$G(\Gamma,f)$ is isomorphic to the quotient group of the fundamental
group $\pi_{1}(\Sigma)$ by the free homotopy class $[\gamma]$.
\end{thm}

This leads us to the natural source where the labels for vertices of
$\Gamma$ may come from. With every embedding $f$ of $\Gamma$ into
$\Sigma$, one gets a natural labeling of vertices by elements of
$G(\Gamma,f)$.

Thus, curves in $S_{g}$ satisfying the checkerboard colorability
condition naturally obtain the group labeling, which allows one to
define the group bracket and the group delta for these knots.

\section{Unsolved Problems}

The bracket and the Turaev delta for curves in $2$-surfaces can be
treated as follows.

We study curves as conjugacy classes $[\gamma]$ of elements $\gamma$
of the fundamental group for an oriented $2$-surface $S_{g}$. Every
summand of $\Delta$ transform one such conjugacy class $[\gamma]$
into two conjugacy classes $[\gamma_{1}],[\gamma_{2}]$ such that
$\gamma_{1}\cdot \gamma_{2}=\gamma$. Here $\gamma_{1}\cup
\gamma_{2}$ is the result of smoothing of $\gamma$ at some crossing
$c$.

For which groups $G$ can one define a $\Delta$-like comultiplication
map $G\to \Z_{2} G\to G$ where the sum is taken along ``reference
points''?

How to define these ``reference points'' in a way similar to
crossings of a diagram? Possibly, if the comultiplication operation
is expected to be defined in the language of words for some
presentation of $G$, the operation should be defined in terms of
some letters selected in a specific way?

A similar question can be asked about the group bracket and the
parity bracket: we expect a well-defined map from $G$ to the direct
sum of various tensor products $G\otimes G\otimes \dots \otimes G$.

Can we define such a bracket for elements $g\in G$ realizing the
principle $[g]=g$ and solving the word problem/ the conjugacy
problem for some groups?


\begin{thebibliography}{99}



\bibitem[ChM]{ChrismanManturov} M.W.Chrisman, V.O.Manturov, Fibered Knots and Virtual Knots, {\em Journal of Knot Theory and Its
Ramifications}, Vol. 22, No. 12 (2013) 1341003 (23 pages).


 \bibitem[Goldman] {Goldman}
W.~Goldman, ``Invariant functions on Lie groups and Hamiltonian
flows of surface group representations'', {\it Inventiones
Mathematicae} {\bf 85},  263--302 (1986).


\bibitem[IMN]{IMN} D.P.Ilyutko, V.O.Manturov, I.M.Nikonov, {\em Parity in Knot Theory and Graph-Links}, CMFD, 41 (2011),  3Ц163


\bibitem[FM]{FedoseevManturov} {\em Invariants of homotopy classes of curves and graphs on 2-surfaces}, Fundam. Prikl. Mat., 2013,  Volume 18,   Issue 4,
Pages 89Ц105

\bibitem[KM]{KrasnovManturov} V.A.Krasnov, V.O.Manturov, {\em Graph-Valued Invariants of Virtual and
Classical Links and Minimality Problem} Journal of Knot Theory and
Its Ramifications Vol. 22, No. 12 (2014) 1341006 (14 pages)

\bibitem[Ma1]{Parity} V.O.Manturov, {\em Parity in Knot Theory}, Mat.
Sbornik, 201:5 (210), pp. 65-110.


\bibitem[Ma2]{Cobordisms}
V.O.Manturov, {\em Parity and Cobordisms of Free Knots}, Mat.
Sbornik, 203: 2 (2012), pp. 45-76.

\bibitem[Ma3]{Doklady} V.\,O.~Manturov,
``An almost classification of free knots'', {\em Doklady
Mathematics}, \textbf{88}:2 (2013), 556--558 (Original Russian Text
in {\em Doklady Akademii Nauk}, \textbf{452}:4 (2013), 371--374).

\bibitem[Ma4]{NewParity} V.O.Manturov, New Parities and Coverings over
Free Knots, (in Russian) Izvestiya RAN, Submitted

\bibitem[Ma5]{Projection} V.O.Manturov,Parity and Projection from virtual
knots to classical knots, Journal of Knot Theory and Its
Ramifications Vol. 22, No. 9 (2013) 1350044 (20 pages).

\bibitem[Ma6]{Crossing} V.O.Manturov, Virtual Crossing Numbers for
Virtual knots, Journal of Knot Theory and Its Ramifications Vol. 21,
No. 13 (2012) 1240009 (13 pages)



 \bibitem[Tu1]{Turaev}
V.~G.~Turaev, ``Topology of words'', {\it Proc.\ Lond.\ Math.\ Soc.}
{\bf 95}:3, 360--412 (2007).

 \bibitem {Turaev2}
V.~G.~Turaev, ``Algebras of loops on surfaces, algebras of knots,
and quantization'', {\it Braid Group, Knot Theory and Statistical
Mechanis (C.~N.~Yang and M.~L.~Ge, eds), Math.\ Phys.} {\bf 9}, {\it
World Sci.\ Publ., Signapore}, 59--95 (1989).

\end{thebibliography}
\end{document}